\date{}
\title{Near-optimal separators in string graphs}

\newcommand{\cmt}[1]{\ifhmode\newline\fi{\sf *** \ \ #1 \\}}

\author{
{\sc Ji\v{r}\'{\i} Matou\v{s}ek}\thanks{Supported
by the  ERC Advanced Grant No.~267165
and by GRADR Eurogiga GIG/11/E023.}
\\
   {\footnotesize Department of Applied Mathematics and}\\[-1.5mm]
   {\footnotesize Institute of Theoretical Computer Science (ITI)}\\[-1.5mm]
   {\footnotesize  Charles University, Malostransk\'{e} n\'{a}m. 25}\\[-1.5mm]
{\footnotesize  118~00~~Praha~1,
   Czech Republic, and}\\
{\footnotesize    Institute of  Theoretical Computer Science}\\[-1.5mm]
{\footnotesize    ETH Zurich,
      8092 Zurich, Switzerland}
}

\documentclass[11pt]{article}

\usepackage{a4wide}
\usepackage{latexsym}
\usepackage{amssymb,amsmath,amsthm}
\usepackage{graphicx}
\usepackage{url}

\newtheorem{theorem}{Theorem}

\newtheorem{lemma}[theorem]{Lemma}

\newcommand{\heading}[1]{\vspace{1ex}\par\noindent{\bf #1}}

\makeatletter
\newcommand{\ProofEndBox}{{\ifhmode\unskip\nobreak\hfil\penalty50 \else
          \leavevmode\fi\quad\vadjust{}\nobreak\hfill$\Box$
            \finalhyphendemerits=0 \par}}
\makeatother

\newcommand{\R}{{\mathbb{R}}}

\newcommand\PP{\mathcal{P}}

\def\:{\colon}

\DeclareMathOperator{\vcong}{vcong}

\DeclareMathOperator{\conge}{cong}
\DeclareMathOperator{\crr}{cr}
\DeclareMathOperator{\pcr}{pcr}

\long\def\onefigure#1#2{
\begin{figure*}[tbp]
\begin{center}
#1
\end{center}
\caption{#2}
\end{figure*}
}

\def\immediateFigure#1{%
\smallskip\begin{center}#1\end{center}\smallskip }

\newcommand{\labfig}[2]  
{\onefigure{\mbox{\includegraphics{Figures/#1}}}{\label{f:#1} #2} }

\newcommand{\labfigw}[3]  
{\onefigure{\mbox{\includegraphics[width=#2]{Figures/#1}}}{\label{f:#1} #3}}

\newcommand{\immfig}[1]  
{\immediateFigure{\mbox{\includegraphics{Figures/#1}}}}

\newcommand{\immfigw}[2] 
{\immediateFigure{\mbox{\includegraphics[width=#2]{Figures/#1}}}}

\begin{document}

\maketitle

\begin{abstract} Let $G$ be a string graph (an intersection graph of continuous
arcs in the plane) with $m$ edges. 
Fox and Pach proved that $G$ has a separator consisting
of $O(m^{3/4}\sqrt{\log m})$ vertices, and they conjectured that
the bound of $O(\sqrt m)$ actually holds. We obtain 
separators with $O(\sqrt m \,\log m)$ vertices.
\end{abstract}


Let $G=(V,E)$ be a graph with $n$ vertices. A \emph{separator} in $G$
is a set $S\subseteq V$ of vertices such that  there is a partition
$V=V_1\cup V_2\cup S$ with $|V_1|,|V_2|\le \frac23n$ and no edges
connecting $V_1$ to $V_2$.
 The graph $G$
is a \emph{string graph} if it is an intersection graph of curves in the plane,
i.e., if there is a system $(\gamma_v:v\in V)$ of curves (continuous arcs)
such that $\gamma_u\cap\gamma_v\ne\emptyset$ iff $\{u,v\}\in E(G)$ or $u=v$.

Fox and Pach \cite{fox2010separator} proved that 
every string graph has a separator with
$O(m^{3/4}\sqrt{\log m})$ vertices, where $m$ is the number of 
edges of $G$. 

We should mention that they
actually proved the result for the weighted case, where
each vertex $v\in $V has a positive real weight,
and the size of the components of $G\setminus S$ is measured by
the sum of vertex weights (while the size of $S$ is still measured
as the number of vertices). Our result can also be extended to
the weighted case, either by deriving it from the unweighted case
along the lines of \cite{fox2010separator}, or by using appropriate
vertex-weighted versions (available in the cited sources)
of the tools used in the proof. However, for simplicity, we stick
to the unweighted case in this note.

Pach and Fox conjectured that string graphs actually have
separators of size $O(\sqrt m\,)$  (which, if true,
would be asymptotically optimal in the worst case).
Earlier, in \cite{FoPa-septur}, they proved some special cases 
of this conjecture: most notably, if every two curves $\gamma_u,\gamma_v$
in the string representation intersect in at most $k$ points,
where $k$ is a constant.
As they kindly informed me in February 2013, they also have an (unpublished)
proof of existence of separators of size $O(\sqrt n\,)$ in string
graphs with maximum degree bounded by a constant.
Here we obtain the following result.

\begin{theorem}\label{t:} Every string graph $G$ with $m\ge 2$ edges has
a separator with $O(\sqrt m\,\log m)$ vertices.
\end{theorem}

Clearly, we may assume $G$ connected, and then 
the theorem immediately follows from Lemmas~\ref{l:1} and~\ref{l:2} below.
Lemma~\ref{l:1} combines the considerations of \cite{fox2010separator}
with those of \cite{kolman2004crossing} and adjusts them for
vertex congestion instead of edge congestion. 
Lemma~\ref{l:2} replaces an approximate duality between
sparsity of edge cuts and edge congestion due to Leighton and
Rao \cite{LeightonRao} used in \cite{kolman2004crossing} with an
approximate duality between sparsity of vertex cuts and
vertex congestion, which is an immediate consequence
of the results of Feige et al.~\cite{feige2008improved}.

Fox and Pach \cite{foxpach-stringappl} obtained several
interesting applications of Theorem~\ref{t:}. Here we mention
yet another consequence.

\heading{Crossing number versus pair-crossing number. } The \emph{crossing
number} $\crr(G)$ of a graph $G$ is the minimum possible number of
edge crossings in a drawing of $G$ in the plane, while the
\emph{pair-crossing number} $\pcr(G)$ is the minimum possible
number of pairs of edges that cross in a drawing of~$G$.

One of the most tantalizing questions in the theory of graph drawing
is whether $\crr(G)=\pcr(G)$ for all graphs $G$
\cite{PachToth-cr}, and in the absence of
a solution, researchers have been trying to bound $\crr(G)$ from above
by a function of $\pcr(G)$. The strongest result so far by 
T\'oth \cite{toth430better} was $\crr(G)=O( p^{7/4}(\log p)^{3/2})$,
where $p=\pcr(G)$. It is based on the Fox--Pach separator theorem
for string graphs discussed above, and by replacing their bound
by Theorem~\ref{t:} in T\'oth's proof, one obtains the improved 
estimate $\crr(G)=O(p^{3/2}\log^2 p)$.

%

\heading{Vertex congestion in string graphs. }
Let $\PP$ denote the set of all paths in $G$, and for each
pair $\{u,v\}\in {V\choose 2}$ of vertices, let $\PP_{uv}
\subseteq \PP$ be all paths from $u$ to $v$. An \emph{all-pair
unit-demand multicommodity flow} in $G$ is a mapping
$\varphi\:\PP\to [0,1]$ such that $\sum_{P\in\PP_{uv}}\varphi(P)=1$
for every $\{u,v\}\in {V\choose 2}$. The \emph{congestion} $\conge(w)$ of
a vertex $w\in V$ under $\varphi$ is the total flow through $w$
where, for conformity with 
\cite{feige2008improved}, we count only half of the flow 
through a path $P$ if $w$ is one of the endpoints of $P$.
That is,
\[
\conge(w)=\sum_{P\in\PP: w{\rm\ internal\ vertex\ of\ }P} \varphi(P)+
\frac 12 \sum_{P\in\PP: w{\rm~endpoint~of~}P} \varphi(P).
\]
We define $\vcong(G):=
\min_\varphi \max_{w\in V} \conge(w)$, where the minimum is over
all all-pair unit-demand multicommodity flows.\footnote{It is well known, 
and easy  to check by a compactness argument, that $\min$ is attained.}

\begin{lemma}\label{l:1} 
If $G$ is a connected string graph, then
$\vcong(G)\ge c n^2/\sqrt m$ (for a suitable constant $c>0$).
\end{lemma}

\begin{proof} Let $\varphi$ be a flow for which $\vcong(G)$ is attained,
and let $(\gamma_v:v\in V)$ be a string representation of $G$. We construct
a drawing of $K_V$, the complete graph on the vertex set $V$, as follows.

We draw each vertex $v\in V$ as a
point $p_v\in\gamma_v$, in such a way that
all the $p_v$ are distinct. 

For every edge
$\{u,v\}\in {V\choose 2}$ of the complete graph, we pick a path 
$P_{uv}$ from $\PP_{uv}$ at random, where each $P\in \PP_{uv}$ is
chosen with probability $\varphi(P)$, the choices being independent
for different $\{u,v\}$. 
Let us enumerate the vertices along $P_{uv}$
as $v_0=u,v_1,v_2,\ldots,v_k=v$. Then we draw the edge $\{u,v\}$
of $K_V$ in the following manner: We start at $p_u$, follow
$\gamma_{u}$ until some (arbitrarily chosen) intersection with 
$\gamma_{v_1}$, then we follow $\gamma_{v_1}$ until some intersection
with $\gamma_{v_2}$, etc., until we reach $\gamma_v$ and $p_v$ on it.

Let us estimate the expected number of pairs $\{\{u,v\},\{u',v'\}\}$
of edges of $K_V$
that intersect in this drawing. 

The drawings of $\{u,v\}$ and $\{u',v'\}$ may intersect only if there are
vertices $w\in P_{uv}$ and $w'\in P_{u'v'}$ such that
$\gamma_w\cap \gamma_{w'}\ne \emptyset$, i.e., $\{w,w'\}\in E(G)$ or $w=w'$.
For every choice of $\{w,w'\}\in E(G)$ or $w=w'\in V$, 
the expected number of pairs $\{P_{uv},P_{u'v'}\}$ with $w\in P_{uv}$ 
and $w'\in P_{u'v'}$
is easily seen to be bounded above by $4\vcong(G)^2$
(using linearity of expectation and independence).
Thus, the total expected number of intersecting pairs of edges of $K_V$
is at most $4(m+n)\vcong(G)^2\le 4(2m+1)\vcong(G)^2$. 

At the same time, it is well known that $\pcr(K_V)=\Omega(n^4)$,
i.e., any drawing of $K_V$ has $\Omega(n^4)$ intersecting pairs of edges
(see, e.g., \cite[Thm.~3]{PachToth-cr}).
So $m \vcong(G)^2=\Omega(n^4)$ and 
the lemma follows.
\end{proof}

\heading{Vertex congestion and separators. } Let us define 
$\vcong^*(G):=\min\{\vcong(H):
H\mbox{ is an induced subgraph of }G\mbox{ on at least $\frac23n$ vertices}\}$.

\begin{lemma}\label{l:2} 
Every graph $G$ on $n$ vertices has a vertex separator
with $O((n^2 \log n)/\vcong^*(G))$ vertices.
\end{lemma}

\begin{proof} The proof goes in the following steps, all of them
contained in \cite{feige2008improved} (also see
\cite{biswal2010eigenvalue}, especially Sec.~5.2 there, for a similar
use of \cite{feige2008improved}).
\begin{enumerate} 
\item Let $s\:V\to [0,\infty)$ be an assignment of real weights to
the vertices of $G$, let the weight of an edge $e=\{u,v\}\in E(G)$
be $(s(u)+s(v))/2$, and let $d_s$ be the shortest-path pseudometric
in $G$ with these edge weights. By the duality of linear
programming, it is easy to derive (see \cite[Sec.~4]{feige2008improved})
\[
\textstyle
\frac1{\vcong(G)}=
\min \bigl\{\sum_{v\in V}s(v):\sum_{\{u,v\}\in{V\choose 2}} d_s(u,v)=1\bigr\}.
\]
\item Let $s^*$ be a vertex weighting attaining the minimum in the last formula.
By using a famous result of Bourgain suitably, see \cite[Theorem~3.1]{feige2008improved}, we get that there exists a function $f\:V\to \R$ that is
1-Lipschitz w.r.t. $s^*$, i.e., $|f(u)-f(v)|\le d_{s^*}(u,v)$ for all
$u,v\in V$, and such that $\sum_{\{u,v\}\in{V\choose 2}} |f(u)-f(v)|=
\Omega((\sum_{\{u,v\}\in{V\choose 2}}d_{s^*}(u,v))/\log n)=
\Omega(1/\log n)$.
\item Let $(A,B,S)$ be a partition of the vertex set of a graph $G$
into three disjoint subsets with $A\ne\emptyset\ne B$ and
no edges between $A$ and $B$. Let the \emph{sparsity} of $(A,B,S)$ be
$
\frac{|S|}{|A\cup S|\cdot|B\cup S|}$. By
\cite[Lemma~3.7]{feige2008improved}, given a function $f$ as above for $G$,
there exists a partition $(A,B,S)$ of the vertex set with sparsity
$O((\sum_{v\in V} s^*(v))\log n)=O((\log n)/\vcong(G))$.
\item 
A standard procedure, starting with $G$ and repeatedly finding
a sparse partition until the size of all components drops below
$\frac 23 n$ (see, e.g., \cite[Sec.~6]{feige2008improved}), then
finds a separator of size $O((n^2\log n)/\vcong^*(G))$ in $G$ as claimed.
\end{enumerate}
\end{proof}

\heading{Remark. } Although Lemma~\ref{l:2} is tight for arbitrary graphs,
a possible way towards proving the Fox--Pach conjecture, separators
for string graphs of size $O(\sqrt m\,)$, would be removing the $\log n$
factor in Lemma~\ref{l:2} under the assumption that $G$ is a string graph.
More concretely, the improvement might be achievable in item~2 of the
proof above: indeed, if $G$ is a planar graph or, more generally,
belongs to a minor-closed class of graphs with a forbidden minor,
then, in the setting of item~2, the 1-Lipschitz $f$ 
can even be made to satisfy $\sum_{\{u,v\}\in{V\choose 2}} |f(u)-f(v)|=
\Omega(1)$ \cite{Rabinovich} (also see \cite[Thm.~3.2]{feige2008improved}).
Thus, a similar improvement for string graphs is perhaps not out of reach.

\heading{Acknowledgment. } I would like to thank Jacob Fox
and J\'anos Pach, as well as an anonymous referee,
 for very useful comments.

\bibliographystyle{alpha}
\bibliography{../../bib/cg.bib}

\end{document}